\documentclass[11pt,reqno]{amsart}
\usepackage{amsmath}
\usepackage{amssymb}
\usepackage{amsthm}
\usepackage{enumerate}
\usepackage[mathscr]{eucal}

\usepackage{mathtools}
\usepackage{microtype}

\usepackage{times}
\usepackage[english]{babel}

\usepackage{amsmath,amssymb,latexsym,textcomp,mathrsfs}
\usepackage[all]{xy}
\usepackage{graphicx}
\usepackage{bm,amsmath, amsthm, amssymb, amsfonts}
\usepackage{times}
\usepackage[english]{babel}

\setbox0=\hbox{$+$}
\newdimen\plusheight
\plusheight=\ht0
\def\+{\;\lower\plusheight\hbox{$+$}\;}  

\setbox0=\hbox{$-$}
\newdimen\minusheight
\minusheight=\ht0
\def\-{\;\lower\minusheight\hbox{$-$}\;}

\setbox0=\hbox{$\cdots$}
\newdimen\cdotsheight
\cdotsheight=\plusheight
\def\cds{\lower\cdotsheight\hbox{$\cdots$}}

\newcommand{\norm}[1]{\left\lVert#1\right\rVert}

\numberwithin{equation}{section}
\newtheorem{thm}{Theorem}[section]
\newtheorem{lem}[thm]{Lemma}

\newtheorem{cor}{Corollary}

\theoremstyle{definition}
\newtheorem{defn}{Definition}[section]

\newtheorem{exmp}{Example}[section]

\theoremstyle{remark}
\newtheorem{rem}{\bf{Remark}}
\newtheorem{note}{\bf{Note}}
\numberwithin{equation}{section}

\begin{document}

\setcounter {page}{1}
\title{ Rough $I$-convergence in cone metric spaces.}

\author{Amar Kumar Banerjee and Anirban Paul}

\address[A.K.Banerjee]{Department of Mathematics, The University of Burdwan, Golapbag, Burdwan-713104, West Bengal, India.}

\email{ akbanerjee1971@gmail.com, akbanerjee@math.buruniv.ac.in }

\address[A.Paul]{Department of Mathematics, The University of Burdwan, Golapbag, Burdwan-713104, West Bengal, India.}
         
\email{ paulanirban310@gmail.com}

\begin{abstract} Here we have studied the notion of rough $I$-convergence as an extension of the idea of rough convergence in a cone metric space using ideals. We have further introduced the notion of rough $I^*$-convergence of sequences in a cone metric space to find the relationship between rough $I$ and $I^*$-convergence of sequences.
\end{abstract}
\maketitle
\author{}


\textbf {(2010) AMS subject classification :} 40A05, 40A99. \\
\section{Introduction}
The concept of statistical convergence was given independently by Fast\cite{f} and Steinhaus\cite{q} as a generalization of ordinary convergence of real sequences. Lot of devolopments have been made in this area after the remakable works done in \cite{g,new4, p}. The idea of $I$-convergence was introduced by Kostyrko et al.\cite{h} as a generalization of statistical convergence using the idea of the ideal $I$ of subsets of the set of natural numbers. Later many more works have been carried out in this direction \cite{B, 2,4,3,6,t**,i,j}.\\
In 2001, Phu \cite{R2} introduced the idea of rough convergence and rough Cauchyness of sequences in a finite dimensional normed space. Later in 2003, he studied the same in an infinite dimensional normed space \cite{R4}. Using these ideas Aytar\cite{R1}, in 2008 gave the concept of rough statistical convergence of a sequence. In 2014, using the concepts of $I$-convergence and rough convergence D\"undar et al.\cite{R7} introduced the notion of rough $I$-convergence. Many more works have been done in different direction \cite{R6, new1, new2} by several authors using this idea given by Phu \cite{R2}. \\
The idea of cone metric spaces was given by Huang and Xian \cite{k}. In their paper they have replaced the distance between two points by the elements of a real Banach space. Obviously such space is a generalization of the notion of an ordinary metric space. Since then many more works have been carried out specially in the field of summability theory.\\
In \cite{5} Banerjee and Mondal have studied the idea of rough convergence of sequences in a cone metric space. In our paper we have studied the notion of rough $I$-convergence of sequences in a cone metric space and examine how far several results as valid in \cite{5} are affected. Also we have introduced here the idea of rough $I^*$-convergence of sequence in a cone metric space and obtain the relations between rough $I$ and $I^*$-convergence of sequences.
\section{Preliminaries}
We give the basic ideas of statistical convergence and then ideal convergence of real sequences and few definitions and notions related to these ideas are also furnished which will be needed in the sequel.\\
Let $K$ be a subset of the set of positive integers $\mathbb{N}$. Then the natural density of $K$ is given by $\delta(K)=\displaystyle{\lim_{n\to \infty}}\frac{|K_n|}{n}$, where $K_n=\left\{k\in K : k\leq n\right\}$ and $|K_n|$ denotes the number of elements in $K_n$.

\begin{defn}\cite{f}
A sequence $x=\left\{x_n\right\}$ of real numbers is said to be statistically convergent to $x$ if for any $\varepsilon >0$ the set $\left\{n\in \mathbb{N}: |x_n - x| \geq \varepsilon\right\}$ have natural density zero.
\end{defn}
Note that the idea of statistical convergence of sequences is a generalization ordinary convergence.\\
 A family $I \subset 2^\mathbb{N}$ of subsets of $X$ is said to be an ideal \cite{t} in $\mathbb{N}$ if the following conditions holds:
\\
$(i)\: A, B \in I \Rightarrow A\cup B\in I$\\
$(ii) \: \:A\in I, B\subseteq A \Rightarrow B\in I$

$I$ is called a non-trivial ideal in $\mathbb{N}$ if $I\neq \left\{\phi\right\}$  or $\mathbb{N}\notin I$. A non-trivial ideal in $\mathbb{N}$ is said to be admissible if $\left\{n\right\}\in I$ for each $n\in \mathbb{N}$.

Clearly if $I$ is a nontrivial ideal in $\mathbb{N}$ then the family of sets $F(I)=\left\{M\subset \mathbb{N} : \text{there exists}\: A\in I, M=\mathbb{N}\setminus A\right\} $ is a filter in $\mathbb{N}$. It is called the filter associated with the ideal $I$ .\\ 
An admissible ideal $I\subset 2^\mathbb{N}$ is said to satisfy the condition (AP) \cite{h} if for any sequence $\left\{A_1, A_2, \cdots \right\}$ of mutually disjoint sets in $I$, there is a sequence $\left\{B_1, B_2,\cdots\right\}$ of subsets of $\mathbb{N}$ such that $A_i \Delta B_i$ ($i=1,2,\cdots )$ is finite and $B=\displaystyle{\cup_{j\in\mathbb{N}}}B_j\in I$.
\begin{defn}\cite{h}
A sequence $x=\left\{x_n\right\}$ of real numbers is said to be $I$-convergent to $x$ if for any $\varepsilon >0$ the set $A(\varepsilon)=\left\{n\in \mathbb{N} : |x_n - x|\geq \varepsilon \right\}\in I$.
\end{defn}
\begin{defn}\cite{R7}
A sequence $x=\left\{x_n\right\}$ in a normed linear space is said to be $I$-bounded if there exists a $M(>0)\in \mathbb{R}$ such that the set $\left\{n\in \mathbb{N}: \norm{x_n}\geq M\right\}\in I$.
\end{defn}
We now give the idea of a cone metric space \cite{k} as follows:\\
Let $E$ be a real Banach space and $P\subset E$. Then $P$ is called a cone if and only if the following are satisfied:\\
$(i)$  $P$ is closed, non-empty and $P\neq \left\{0\right\}$ ($0$ be the zero element of $E$).\\
$(ii)$ $a, b\in \mathbb{R}$ and $a, b\geq 0$ then $x, y\in P$ implies $ax + by \in P$.\\
$(iii)$  $x\in P$ and $- x\in P$ implies that $x=0$ (zero element of $E$).

\indent Let $E$ be a real Banach space and $P$ be a cone in $E$. Then a partial ordering $\leq$ with respect to $P$ can be defined by $x\leq y$ if and only if $ y - x \in P$, whereas $x<y$ indicates $x\leq y$ and $x\neq y$, also $x<<y$ stands for $y -x \in int P$,\: $int P$ denotes the interior $P$.\\
The cone $P$ is called normal if there exists a  positive real number $K$ such that for all $x, y \in E$, $0\leq x\leq y$ implies $\norm{x}\leq K\norm{y}$. \\
The least positive number satisfying above is called the normal constant of $P$. In \cite{lastadd}, it is known that any cone metric space is a first countable Hausdorff topological space with topology induced by the open balls defined naturally for each element $z\in X$ and for every $(0<<)c\in E$.

\begin{defn}\cite{k}
Let $(X, d)$ be a cone metric space. A sequence $x=\left\{x_n\right\}$ in $X$ is said to be convergent to $x$ if for any $c\in E$ with $0<<c$ there is $N\in \mathbb{N}$ such that $d(x_n,x)<<c$ for all $n> N$.
\end{defn}
\begin{lem}\cite{k}\label{correction}
Let $(X, d)$ be a cone metric space and $P$ be a normal cone with normal constant $K$. Let $\left\{x_n\right\}$ be a sequence in $X$. Then $\left\{x_n\right\}$ converges to $x$ if and only if $d(x_n, x)\rightarrow 0$ as $n\rightarrow \infty$
\end{lem}
\begin{lem}\cite{lemma}\label{khani}
Let $(X, E)$ be a cone space with $x\in P$ and $y\in int P$. Then one can find $n\in\mathbb{N}$ such that $x<< ny$.
\end{lem}
Now we recall some useful results from \cite{5}.
\begin{thm}\cite{5}
Let $E$ be a real Banach space with cone $P$. If $x_0\in int P$ and $\alpha(>0)\in \mathbb{R}$ then $\alpha x_0\in int P$
\end{thm}
\begin{thm}\label{rahulda}\cite{5}
Let $E$ be a real Banach space and $P$ be a cone in $E$. If $x_0\in P$ and $y_0\in int P$ then $x_0 + y_0\in int P$.
\end{thm}
\begin{cor}\cite{5}
If $x_0, y_0\in int P$ then $x_0 + y_0\in int P$.
\end{cor}
\begin{thm}\cite{5}
Let $E$ be a real Banach space with cone $P$, then $0\notin int P$ \:($0$ be the zero element of $E$).
\end{thm}

\begin{defn}\cite{5}
Let $\left\{x_n\right\}$ be a sequence in a cone metric space $(X, d)$. A point $c\in X$ is said to be a cluster point of $\left\{x_n\right\}$ if for any $(0<<)\varepsilon$ in $E$ and for any $k\in\mathbb{N}$, there exists a $k_1\in \mathbb{N}$ such that $k_1> k$ with $d(x_{k_1}, c)<< \varepsilon$.
\end{defn}
The definition of $I$-convergent and $I^*$-convergent of a sequence in a cone metric is as follows:
\begin{defn}\cite{sudip}
Let $(X, d)$ be a cone metric space. A sequence $x=\left\{x_n\right\}$ in $X$ is said to be $I$-convergent to $x$ if for any $c\in E$ with $0<<c$ the set $\left\{n\in \mathbb{N} : c - d(x_n, x)\notin int P\right\}\in I$.
\end{defn}
\begin{defn}\cite{sudip}
Let $(X, d)$ be a cone metric space. A sequence $x=\left\{x_n\right\}$ in $X$ is said to be $I^*$-convergent to $x$ if and only if there exists a set $M\in F(I)$, $M=\left\{m_1< m_2<\cdots<m_k< \cdots\right\}$ such that $\left\{x_n\right\}_{n\in M}$ is convergent to $x$ i.e., for any $c\in E$ with $0<<c$, there exists $p\in \mathbb{N}$ such that $c - d(x_{m_k}, x)\in int P$ for all $k\geq p$.
\end{defn}
\begin{defn}\cite{R2}
Let $x=\left\{x_n\right\}$ be a sequence in a normed linear space $(X, \norm{.})$ and $r(\geq 0)\in \mathbb{R}$. Then $\left\{x_n\right\}$ is said to be rough convergent of roughness degree $r$ to $x$ if for any $\varepsilon >0$ there exists a natural number $N$ such that $\norm{x_n - x}< r + \varepsilon$ for all $n\geq k$. 
\end{defn}
 For $r=0$ we obtain the ordinary convergence of sequences.
\begin{defn}\cite{R7}
A sequence $x=\left\{x_n\right\}$ in a normed linear space is said to be rough $I$-convergent of roughness degree $r$ to $x^*$ for some $r\geq 0$ if for any $\varepsilon>0$ the set $\left\{n\in \mathbb{N} : \norm{x_n - x^*}\geq r + \varepsilon\right\}\in I$.\\
We denote this by $x_n\xrightarrow{r - I} x^*$.
\end{defn}
The definition of $I$-bounded sequence in a normed linear space has been given in \cite{R7} as follows:
\begin{defn}\cite{R7}
A sequence $\left\{x_n\right\}$ is said to be $I$-bounded if there exists a positive real number $M$ such that $\left\{n\in \mathbb{N} : || x_n||\geq M\right\}\in I$.
\end{defn}
\begin{defn}\cite{5}
A sequence $\left\{x_n\right\}$ in a cone metric space is said to be bounded if for any fixed $x\in X$ there exists a $(0<<)M\in E$ such that $d(x_n, x)<< M$ for all $n\in \mathbb{N}$.
\end{defn}
Now we recall the definition of rough convergence in  cone metric space from \cite{5}.
\begin{defn}\cite{5}
Let $(X, d)$ be a cone metric space. A sequence $x=\left\{x_n\right\}$ in $X$ is said to be rough convergent of roughness degree $r$ to $x$ for some $r\in E$ with $0<< r$ or $r=0$ if for any $\varepsilon$ with $(0<<)\varepsilon$ there exists a $m\in \mathbb{N}$ such that $d(x_n, x) << r +\varepsilon $ for all $n\geq m$.\\
We denote this by $x_n\xrightarrow{r }x$.
\end{defn}

\section{Main Results }

Throughout our discussion $(X,d)$ will always stands for a cone metric space where $d : X \times X \mapsto E$ is the cone metric and $E$ being a real Banach space. $I$ be a admissible ideal, $\mathbb{N}$ and $\mathbb{R}$ stands for the set of natural numbers and the set of real numbers respectively. $A^\complement$ denotes the complement of the set $A$ unless otherwise stated.

\begin{defn}
Let $(X, d)$ be a cone metric space. A sequence $x=\left\{x_n\right\}$ in $X$ is said to be rough $I$-convergent of roughness degree $r$ to $x^*\in X$  for some $r\in E$ with $0<<r$ or $r=0$ if for any $(0<<)\varepsilon \in E$ the set $A(\varepsilon)=\left\{n\in \mathbb{N}: (r + \varepsilon -d(x_n, x^*) )\notin int P\right\}\in I$.
\end{defn}
We denote this by $x_n\xrightarrow {r - I} x^*$. For $r=0$ the definition reduces to the definition of $I$-convergence of sequence in a cone metric space. If a sequences $x=\left\{x_n\right\}$ is rough $I$-convergent of roughness degree $r$ to $x^*\in X$ then $x^*$ is called the rough $I$-limit of $x=\left\{x_n\right\}$. In general, the rough $I$-limit of a sequence $x=\left\{x_n\right\}$ is not unique which can be seen from the next  example. So the set of all rough $I$-limits of a sequence $x=\left\{x_n\right\}$ denoted by $I-LIM^r x$ is called the rough $I$-limit set of a sequence $x=\left\{x_n\right\}$ i.e., $I-LIM^r x :=\left\{x^*\in X : x_n\xrightarrow {r - I} x^*\right\}$. \\
Therefore, a sequence $x=\left\{x_n\right\}$ is said to be rough $I$-convergent in a cone metric space if $I-LIM^r x\neq \phi$
\begin{exmp}\label{ex1}
Let $X=\mathbb{R}$, $E=\mathbb{R}^2$, $P=\left\{(x, y)\in E : x, y \geq 0\right\} \subset E$ and $d : X \times X \mapsto E$ be such that $d(x, y)=(|x - y|, |x -y|)$. Then $(X, d)$ is a cone metric space. Now let us consider the ideal in $\mathbb{N}$ which consists of sets whose natural density are zero i.e., $I=I_d$. Now, let us consider the sequence $x=\left\{x_n\right\}$ in $X$ defined by $x_n=\begin{cases}
(-1)^n , \: &
\text{if} \: n\neq k^2(\text{where} \:k \in \mathbb{N}) \\
n, \: &
\text{otherwise}
\end{cases}$. Now we can see that for any $r=(r_1, r_2)\in E$ with $0<< r$ , if $\min( {r_1, r_2}) =r^*$ and $r^* \geq 1$ then $I-LIM^r x=[-(r^* - 1), (r^* - 1)]$ if $r^*\geq 1$, since for any $x\in [-(r^* - 1), (r^* - 1)]$ and $r^*\geq 1$ we have $\{n\in \mathbb{N} : (r + \varepsilon - d(x_n, x)\notin int P\}\subset \{1^2, 2^2, 3^2, \cdots\}$, therefore $\{n\in \mathbb{N} : (r + \varepsilon - d(x_n, x)\notin int P\}\in I$ and if $r^* < 1$ or $r =0$ then $I-LIM^r x=\phi$. Also since the sequence is unbounded therefore $LIM^r x=\phi$, for any $r$. 
\end{exmp} 
\begin{note}\label{note1}
From the above example we can see that in general $I-LIM^r x\neq \phi$ does not imply that $LIM^r x \neq \phi$. But since $I$ is an admissible ideal therefore $LIM^r x\neq \phi$ implies $I-LIM^r x\neq\phi$. That is, if a sequence $x=\left\{x_n\right\}$ in $(X, d)$ is rough convergent of  roughness degree $r$, where $r\in E$ with $0<< r$ or $r=0$,  then it is also rough $I$-convergent of same roughness degree $r$. Therefore if we denote all rough convergence sequences in a cone metric space $(X, d)$ by $LIM^r$ and the set of all rough $I$-convergent sequences by $I-LIM^r$, then we have $LIM^r \subseteq I-LIM^r$.
\end{note}
It is seen in \cite{5} that if a sequence $x=\{x_n\}$ in a cone metric space $(X, d)$ is bounded then $LIM^r x\neq \phi$ for some $(0<<)r \in E$. So in view of note \ref{note1}, the following theorem is evident.

\begin{thm}\label{cor1}
If a sequence $x=\left\{x_n\right\}$ in a cone metric space $(X, d)$ is bounded, then there exists some $r\in E$ with $0<< r$ such that $I-LIM^r x\neq \phi$.
\end{thm}
We recall that a sequence $\{x_n\}$ in a metric space $(X, d)$ is said to be bounded if there exists $x\in X$ and $r> 0$ satisfying $d(x_n ,x)< r \: \text{for all}\: n\in \mathbb{N}$. Using this idea we define $I$-bounded sequence in a cone metric space as follows:  
\begin{defn}
A sequence $x=\left\{x_n\right\}$ in a cone metric space $(X, d)$ is said to be $I$-bounded if there exists a $y\in X$ and $M\in E$ with $0<< M$ such that $\left\{n\in \mathbb{N} : M - d(x_n, y)\not\in int P\right\}\in I$.
\end{defn}
Let $\{x_n\}$ be bounded sequence in a cone metric space $(X, d)$, then there exists $x\in X$ and $M\in E$ with $0<< M$ such that $d(x, x_n) << M$ for all $n\in \mathbb{N}$. This implies that $M- d(x, x_n)\in int P \: \text{for all}\: n\in \mathbb{N}$. So $\{n\in \mathbb{N} : M - d(, x_n)\notin int P\}= \phi\in I$. Hence $\{x_n\}$ is $I$-bounded. But the converse may not be true as seen in the example \ref{ex1}. For, if we choose $y= 2$ and $(0<<)M=(5,6)$ then we get $\{n\in \mathbb{N} : (M - d(x_n, y))\notin int P\}\subset \{1^2, 2^2, 3^2, \cdots\}$, which implies that $\{n\in \mathbb{N} : M - d(x_n, y)\notin int P\}\in I$. So the sequence considered here is $I$-bounded although the sequence is not bounded.\\

From the example {\ref{ex1}  it follows that the reverse implication of the theorem \ref{cor1} is not valid, however the reverse implication is true in case of $I$-boundedness as seen in the following theorem.
\begin{thm}
Let $I$ be an admissible ideal of $\mathbb{N}$. Then a sequence $x=\left\{x_n\right\}$ in $(X, d)$ is $I$-bounded if and only if there exists some $r\in E$ with $0<< r$ or $r=0$ such that $I-LIM^r x \neq \phi$.
\end{thm}
\begin{proof}
Let the sequence $x=\left\{x_n\right\}$ be $I$-bounded. Then there exists a $y\in x$ and $(0<<)r \in E$ such that the set $\left\{n \in \mathbb{N}: r - d(x_n , y) \not \in int P\right\}\in I$. Let $(0<<)\varepsilon\in E$ ( i.e., $ \varepsilon \in int P$ ). Then $\{n\in\mathbb{N} : r + \varepsilon - d(x_n ,y)\notin int P\} \subset \{n\in \mathbb{N} : r - d(x_n, y) \notin int P\}\in I$ ( For, let $n\in \{n\in \mathbb{N}: r + \varepsilon - d(x_n, y)\notin int P\}\Rightarrow r + \varepsilon - d(x_n,y)\notin int P$. So $r - d(x_n, y)\notin int P\Rightarrow n\in \{ r - d(x_n, y)\notin int P\}$ ). Therefore $y\in I-LIM^r x $.\\
Conversely, let $I-LIM^r x\neq \phi$ for some $r\in E$ with $0<< r$ or $r=0$ and $x^*\in I-LIM^r x$. Therefore for any $(0<<)\varepsilon\in E$ (i.e., $\varepsilon \in int P$ ) the set $\left\{n\in \mathbb{N} : r + \varepsilon -d(x_n,x^*)\notin int P\right\}\in I$. Now $r + \varepsilon\in int P$ for any $\varepsilon\in int P$. So taking $M = r +\varepsilon\in int P$ ( i.e., $0<< M$), we have $\left\{n\in \mathbb{N} : M -d(x_n,x^*)\notin int P\right\}\in I$. So the sequence $x=\left\{x_n\right\}$ is $I$-bounded.
\end{proof}
\begin{thm}
An $I$-bounded sequence $x=\left\{x_n\right\}$ in a cone metric space $(X, d)$ always contains a subsequence which is rough $I$-convergent of roughness degree $r$ for some $(0<<)r\in E$.
\end{thm}
\begin{proof}
Let a sequence $x=\left\{x_n\right\}$ in a cone metric space $(X, d)$ be $I$-bounded. Therefore there exists a $z\in X$ and $(0<<)M\in E$ such that the set $\left\{n\in \mathbb{N}: M - d(x_n, z)\notin int P\right\}\in I$. Therefore the set $L=\left\{n\in \mathbb{N}: M - d(x_n, z)\in int P\right\}\in F(I)$. Now if we consider the subsequence $\left\{x_n\right\}_{n\in L}$ then this subsequence is bounded. Now as for any bounded sequence $x=\left\{x_n\right\}$,\: $LIM^r x\neq\phi$ for some $(0<<)r \in E$, so the subsequence $\left\{x_n\right\}_{n\in L}$ is rough convergent of roughness degree $r$ ($(0<<)r \in E$). Hence in view of note \ref{note1} $\left\{x_n\right\}_{n\in L}$ is also rough $I$-convergent of same roughness degree $(0<<)r\in E$. 
\end{proof}
\begin{thm}
Let $\left\{x_n\right\}$ be a sequence in a cone metric space $(X, d)$ which is $I$-convergent to $x$. If $\left\{y_n\right\}$ is another sequence in $(X, d)$ such that $d(x_n, y_n)\leq r$ for some $ (0<<)r \in E$ and for all $n\in \mathbb{N}$. Then $\left\{y_n\right\}$ is rough $I$-convergent of roughness degree $r$ to $x$.
\end{thm}
\begin{proof}
Let $\left\{x_n\right\}$ be a sequence in a cone metric space $(X, d)$ which is $I$-convergent to $x$. Therefore for an $(0<<) \varepsilon\in E$, the set $\left\{n\in \mathbb{N} : \varepsilon - d(x_n, x)\notin int P\right\}\in I$. So $\left\{n\in \mathbb{N} : \varepsilon - d(x_n, x) \in int P\right\}\in F(I)$. Now $d(y_n ,x) \leq d(y_n,x_n) + d(x, x_n)\leq r + d(x_n , x)$. This implies that $r + d(x_n, x) -d(y_n, x) \in P$. Hence if $\varepsilon - d(x_n, x)\in int P$ then $( r + d(x_n, x) -d( y_n, x)) + (\varepsilon - d(x_n, x))= r + \varepsilon - d(y_n, x)\in int P$. Therefore the set $\left\{n\in \mathbb{N} : r + \varepsilon - d(y_n, x) \in int P\right\}\in F(I)$. Thus $\left\{n\in \mathbb{N} : r + \varepsilon - d(y_n, x) \notin int P\right\}\in I$. Hence the results follows.
\end{proof}
\begin{thm}
Let $x=\left\{x_n\right\}$ be a sequence in a cone metric space $(X, d)$ which is rough $I$-convergent of roughness degree $r$ for some $(0<<)r\in E$. Then there does not exists $y, z\in I-LIM^r x$ such that $m r < d(y, z)$, where $m$ is a real number grater than $2$.
\end{thm}
\begin{proof}
Suppose on contrary that there exists  such $y, z\in I-LIM^r x$ for which $mr < d(y, z)$ and $m(\in \mathbb{R}) > 2$. Let $(0<<)\varepsilon $ be arbitrarily chosen in $E$. Now as $y, z\in I-LIM^r x$, so each of the sets $M_1=\left\{n\in \mathbb{N} : r + \frac{\varepsilon}{2} -d(x_n, y)\notin int P\right\}$ and $M_2=\left\{n\in \mathbb{N} : r + \frac{\varepsilon}{2} -d(x_n , z)\notin int P\right\}$ belongs to $I$. Then both of $M_1^\complement$ and $M_2^\complement$ belongs to $F(I)$. Let $p\in M_1^\complement \cap M_2^\complement$. Then $ r +\frac{\varepsilon}{2} -d(x_p, y)\in int P$ and $ r+\frac{\varepsilon}{2} - d(x_p, z)\in int P$. Hence $(r + \frac{\varepsilon}{2} - d(x_p , y) ) + ( r + \frac{\varepsilon}{2} - d(x_p, z)) = 2r + \varepsilon- ( d(x_p, y) + d(x_p , z) ) \in int P$. Now $d(y, z)\leq d(x_p, y) + d(x_p, z)$. So $d(x_p, y) + d(x_p , z) - d(y, z)\in P$. Therefore  $(2r + \varepsilon - (d(x_p, y) + d(x_p, z))) + ( d(x_p, y) + d(x_p, z)  - d(y,z))= 2r + \varepsilon - d(y, z) \in int P$. Again by our assumption $ d(y ,z) - mr \in P$. So $( 2r + \varepsilon - d(y, z)) + ( d(y, z) - mr )=2r + \varepsilon - mr  \in int P$. That is $\varepsilon - r (m - 2) \in int P$. But choosing $\varepsilon= r(m -2)$ we get $0 \in int P$, which is a contradiction. Hence the result follows.
\end{proof}
\begin{thm}\label{th 5}
Let $\left\{x_n\right\}$ be a sequence in $(X, d)$ which is rough $I$-convergent of roughness degree $r$. Then $\left\{x_n\right\}$ is also rough $I$-convergent of roughness degree $r_1$ for any $r_1$ with $r< r_1$.
\end{thm}
\begin{proof}
Proof is trivial and so is omitted.
\end{proof}
In view of the theorem \ref{th 5} we have the following corollary.
\begin{cor}
Let $x = \left\{x_n\right\}$ be a rough $I$-convergent sequence in $(X, d)$ of roughness degree $r$. Then for a $(0<<)r_1$ with $r< r_1$, \: $LIM^r x \subset LIM^ {r_1} x$.
\end{cor}
\begin{defn} (cf. \cite{j} )
A point $c\in X$ is said to be a $I$-cluster point of a sequence $\left\{x_n\right\}$ in $(X, d)$ if for any $(0<<) \varepsilon$ the set $\left\{k\in \mathbb{N} :  \varepsilon - d(x_k, c) \in int P\right\}\notin I$.  
\end{defn}
For $0< < r$ and a fixed $y\in X$, the closed spheres $\overline{B_r(y)}$ and open spheres $B_r(y)$ centred at $y$ with radius $r$ is defined in \cite{5} as follows:\\
$\overline{B_r(y)}=\left\{x\in X : d(x, y) \leq r\right\}$ and $B_r(y)=\left\{x\in X : d(x, y) << r \right\}$.\\
Now we have the following theorems.
\begin{thm}\label{th 3.10}
Let $(X, d)$ be a cone metric space. $c\in X$ and $(0<<)r$ be such that for any $x\in X$ either $d(x, c) \leq r$ or $r<< d(x, c)$. If $c$ is a $I$-cluster point of a sequence $\left\{x_n\right\}$ then $I-LIM^r x \subset \overline{B_r(c)}$.
\end{thm}
\begin{proof}
If possible assume that there exists a $y\in I-LIM^r x$ but $y \notin \overline{B_r(c)}$. Now by our assumption $r << d(y, c)$. Let $(0<<)\varepsilon_1 = d(y,c ) - r$ and hence $d(y, c)= r + \varepsilon _1$. Let $(0<<)\varepsilon = \frac{\varepsilon_1}{2}$. Then we have $ d(y , c) = r + 2\varepsilon$. Also $B_{r + \varepsilon}(y) \cap B_\varepsilon (c)=\phi$. For, if $l\in B_{r + \varepsilon}(y) \cap B_\varepsilon (c)$ then $d(l, y) << r +\varepsilon$ and $ d(l, c) <<\varepsilon$. Thus $ r + \varepsilon  -d (l,y )\in int P$ and $\varepsilon  - d(l, c) \in int P$. Therefore $( r + \varepsilon - d(l, y) ) + ( \varepsilon - d(l, c))= r + 2 \varepsilon - ( d(l, y) + d(l, c)) \in int P\rightarrow (i)$. Now as $d(y, c) \leq d(y, l) + d(l, c)$, therefore $d(y, l) + d(l, c) - d(y, c) \in P\rightarrow (ii)$. Hence from $(i)$ and $(ii)$ we get $ r + 2 \varepsilon - (d(l,y) + d(l, c)) + d(y,l) + d(l, c) -d(y, c)= r + 2\varepsilon - d(y, c)=0\in int P$, a contradiction. Hence $B_{r + \varepsilon}(y) \cap B_\varepsilon (c)=\phi$. As $y \in I-LIM^r x$, so the set $A=\left\{n\in \mathbb{N} : r + \varepsilon - d(x_n, y)\notin int P\right\}\in I$. So the set $A^\complement=\mathbb{N}\setminus A\in F(I)$. Again since $c$ is a $I$-cluster point of $\left\{x_n\right\}$ , so for $0<<\varepsilon$ the set $\left\{k\in \mathbb{N} :  \varepsilon - d(x_k, c)\in int P\right\}\notin I$. Therefore the set $\left\{k\in \mathbb{N} :  \varepsilon - d(x_k, c)\in int P\right\}$ can not be a subset of $A$. For, if $\left\{k\in \mathbb{N} :  \varepsilon - d(x_k, c)\in int P\right\} \subset A$ then we have $\left\{k\in \mathbb{N} :  \varepsilon - d(x_k, c)\in int P\right\}\in I$, which contradicts to the fact that $c$ is a $I$-cluster point of $\left\{x_n\right\}$. We consider an element $m\in A^\complement$. So $m\in \left\{k\in \mathbb{N} :  \varepsilon - d(x_k, c)\in int P\right\}$. Now $m\in A^\complement $ implies $r + \varepsilon - d(x_m, y) \in int P$. Hence $d(x_m, y) << r + \varepsilon$, which implies $x_m\in B_{r + \varepsilon}(y)$. Also $m\in \left\{k\in \mathbb{N} :  \varepsilon - d(x_k, c)\in int P\right\}$ implies $\varepsilon - d(x_m, c)\in int P$. Therefore $d(x_m, c) << \varepsilon$, which further implies that $x_m\in B_\varepsilon(c)$. Thus we see that $x_m\in B_{r + \varepsilon}(y)\cap B_\varepsilon(c)$, which is a contradiction. Hence we can conclude that our assumption is wrong and $y\in\overline{ B_r(c) }$.
\end{proof}
\begin{thm}
Let $x=\left\{x_n\right\}$ be a rough $I$-convergent sequence of roughness degree $r$ in a cone metric space $(X, d)$ and $\left\{y_n\right\}$ be a $I$-convergent sequence in $I-LIM^r x$ which is $I$-convergent to $y$. Then $y\in I-LIM^r x$.
\end{thm}
\begin{proof}
Let $(0<<)\varepsilon$ be given. Since the sequence $\left\{y_n\right\}$ is $I$-convergent to $y$, for $(0<<)\varepsilon$ the set $A=\left\{n\in \mathbb{N} : \frac{\varepsilon}{2} - d(y_n, y)\notin int P\right\}\in I$. So the set $A^\complement=\mathbb{N}\setminus A\in F(I)$. Choose a $p\in A^\complement$. Then $\frac{\varepsilon}{2} - d(y_p, y)\in int P$ and so $d(y_p, y)<<\frac{\varepsilon}{2}\rightarrow (i)$. Also since $\left\{y_n\right\}$ is a sequence in $I-LIM^r x$, let $y_p\in I-LIM^r$. Therefore the set $B=\left\{n\in \mathbb{N} : r + \frac{\varepsilon}{2} - d(x_n, y_p)\notin int P\right\}\in I$. Hence it's complement  $B^\complement=\mathbb{N}\setminus B \in F(I)$. Let us choose an element $ l\in B^\complement(\in F(I))$. Therefore $ r + \frac{\varepsilon}{2} - d(x_l,y_p)\in int P$ and so $d(x_l,y_p) << r + \frac{\varepsilon}{2}\rightarrow (ii)$. Also for all $n\in \mathbb{N}$ we have $d(x_n,y)\leq d(x_n, y_p) + d(y_p, y)$. So $d(x_n, y_p) + d(y_p, y) - d(x_n, y)\in P $, for all $n\in \mathbb{N}$. In particular $d(x_l, y_p) + d(y_p, y) - d(x_l, y)\in P \rightarrow(iii)$. Now by $(i)$ and $(ii)$ using the theorem \ref{rahulda} we get $(\frac{\varepsilon}{2} - d(y_p, y)) + (r + \frac{\varepsilon}{2} - d(x_l, y_p))= r + \varepsilon -(d(y_p, y) + d(x_l, y_p))\in int P\rightarrow(iv)$. Applying again the theorem \ref{rahulda} we get from $(iii)$ and $(iv)$,\: $(d(x_l, y_p) + d(y_p, y) - d(x_l, y)) + ( r + \varepsilon -(d(y_p, y) + d(x_l, y_p)))=r + \varepsilon - d(x_l, y)\in int P$. Now since $l$ is chosen arbitrarily from $B^\complement$, therefore the set $\left\{l\in \mathbb{N} : r + \varepsilon - d(x_l, y)\notin int P\right\}\subset B$ and so $\{l\in\mathbb{N} : r + \varepsilon - d(x_l, y)\notin int P\}\in I$. Hence $y\in I-LIM^r x$.
\end{proof}
\begin{thm}
If $\left\{x_n\right\}$ and $\left\{y_n\right\}$ are two sequence in a cone metric space $(X, d)$ such that for any $(0<<)\varepsilon$ the set $\left\{n\in \mathbb{N} : d(x_n, y_n)> \varepsilon\right\}\in I$. Then $\left\{x_n\right\}$ is rough $I$-convergent of roughness degree $r$ to $x$ if and only if $\left\{y_n\right\}$ is rough $I$-convergent of same roughness degree $r$ to $x$ . 
\end{thm}
\begin{proof}
Let $\left\{x_n\right\}$ be rough $I$-convergent of roughness degree $r$ to $x$. Let $(0<<)\varepsilon$ be given. Then the set $\left\{n\in \mathbb{N} : r +\frac{ \varepsilon}{2} - d(x_n, x)\notin int P\right\}\in I\rightarrow(i)$. Also according to our assumption the set $\left\{n\in\mathbb{N} : d(x_n, y_n)> \frac{\varepsilon}{2}\right\}\in I\rightarrow(ii)$. Now complement of the sets in $(i)$ and $(ii)$ belong to $F(I)$ and hence their intersection belong to $F(I)$. Let us choose an element $k\in \mathbb{N}$ in that intersection. Therefore $r + \frac{\varepsilon}{2} - d(x_k, x) \in int P$ and $d(x_k, y_k)\leq \frac{\varepsilon}{2} \: \text{i.e.,}\:\frac{\varepsilon}{2} - d(x_k, y_k)\in P$. So, $(r + \frac{\varepsilon}{2} - d(x_k , x)) + (\frac{\varepsilon}{2} - d(x_k, y_k)) = r + \varepsilon - ( d(x_k, x) + d(x_k, y_k) )\in int P\rightarrow(iii)$. Also for all $n$ \: $d(y_n,x) \leq d(x_n, y_n) + d(x_n, x)$. That is $d(x_n, y_n) + d(x_n, x) - d(y_n,x) \in P$. In particular $d(x_k, y_k) + d(x_k, x) - d(y_k,x) \in P\rightarrow (iv)$. Hence from $(iii)$ and $(iv)$ we get $( r + \varepsilon - (d(x_k, x) + d(x_k, y_k))) + ( d(x_k, y_k) + d(x_k, x) - d(y_k,x) ) = r + \varepsilon -d(y_k, x)\in int P$. Therefore the set $\left\{n\in \mathbb{N} : r + \varepsilon - d(y_k, x) \notin int P\right\}\in I$, which implies that $\left\{y_n\right\}$ is rough $I$-convergent of roughness degree $r$ to $x$.\\
Converse part can be proved by interchanging the role of $\left\{x_n\right\}$ and $\left\{y_n\right\}$.
\end{proof}
\begin{thm}
Let $\mathcal{C}$ be the set of all $I$-cluster points of a sequence $\left\{x_n\right\}$. Also let $(0<<)r\in E$ be such that for any $x\in X$ and for each $c\in \mathcal{C}$ either $d(x, c)\leq r$ or $r<<d(x , c)$. Then $I-LIM^r x \subset \bigcap_{c\in \mathcal{C}}\overline{B_r(c)} \subset \left\{y\in X : \mathcal{C}\subset \overline{B_r(y)}\right\}$.
\end{thm}
\begin{proof}
From the theorem \ref{th 3.10} we can say that $I-LIM^ r x \subset \bigcap_{c \in \mathcal{C}} \overline{B_r(c)}(\subset \overline{B_r(c)})$. To prove the part $\bigcap_{c\in \mathcal{C}}\overline{B_r(c)} \subset \left\{y\in X : \mathcal{C}\subset \overline{B_r(y)}\right\}$, let us take a $z\in \bigcap_{c\in \mathcal{C}}\overline{B_r(c)} $. So $z\in \overline{B_r(c)}$ for each $c\in \mathcal{C}$ and therefore $d(z, c)\leq r$ for every $c\in \mathcal{C}$. This implies that $c\in \overline{B_r(z)}$ for each $c\in \mathcal{C}$. Thus we get $\mathcal{C} \subset \overline{B_r(z)}$. Hence $\bigcap_{c\in \mathcal{C}}\overline{B_r(c)} \subset \left\{y\in X : \mathcal{C}\subset \overline{B_r(y)}\right\}$. Hence the results follows.  
\end{proof}
\begin{defn}
A sequence $\left\{x_n\right\}$ in a cone metric space $(X, d)$ is said to be rough $I^*$-convergent of roughness degree $r$ to $x$ if there exists a set $M=\left\{m_1< m_2< \cdots< m_k < \cdots\right\}\in F(I)$ such that the subsequence $\left\{x_n\right\}_{n\in M}$ is rough convergent of roughness degree $r$ to $x$ for some $(0<<)r\in E$ or $r=0$. That is for any $\varepsilon$ with $(0<<)\varepsilon$ there exists a $k\in \mathbb{N}$ such that $d(x_{m_p}, x)<< r + \varepsilon $ for all $p\geq k$. Here $x$ is called the rough $I^*$-limit of the sequence $\left\{x_n\right\}$.\\
We denote this by $x_n\xrightarrow{ r - I^*} x$.
\end{defn}
\begin{note}
For $r=0$ we have the ordinary $I^*$-convergence of sequences in a cone metric space. Clearly the rough $I^*$-limit of a sequence in general not unique. We shall denote the set of all rough $I^*$-limits of a sequence $\left\{x_n\right\}$ by $I^*-LIM^r =\left\{x\in X : x_n \xrightarrow { r - I^*} x\right\}$ of roughness degree $r$.
\end{note}
\begin{thm}\label{new}
 If a sequence $x=\left\{x_n\right\}$ is rough $I^*$-convergent of roughness degree $r$ to $x$ then it is also rough $I$-convergent of same roughness degree $r$ to $x$.
\end{thm}
\begin{proof}
Let us assume that the sequence $\left\{x_n\right\}$ is rough $I^*$-convergent of roughness degree $r$ to $x$. Therefore by the definition, there exists a set $M=\left\{m_1< m_2< \cdots< m_k < \cdots\right\}\in F(I)$ such that $\left\{x_n\right\}_{n\in M}$ is rough convergent of roughness degree $r$ to $x$. That is for any $(0<<) \varepsilon$ there exists a $p\in \mathbb{N}$ such that $d(x_{m_k}, x) << r + \varepsilon$ for all $k\geq p$. Now the set $\left\{n\in \mathbb{N} : r + \varepsilon - d(x_n, x) \notin int P\right\} \subset \mathbb{N}\setminus M\cup \{m_1,m_2,\cdots, m_{p -1}\}$. As $\mathbb{N}\setminus M \cup \{m_1,m_2,\cdots, m_{p-1}\}\in I$ therefore the set $\left\{n\in \mathbb{N} : r + \varepsilon - d(x_n , x)\notin int P\right\} \in I$. Hence the sequence $\left\{x_n\right\}$ is rough $I$-convergent of roughness degree $r$ to $x$. This proves our theorem.
\end{proof}
It may happen that a sequence $\{x_n\}$ in a cone metric space $(X, d)$ is rough $I$-convergent of roughness degree $r$ to $x\in X$ without being rough $I^*$-convergent of same roughness degree $r$ to $x$. Following example is such one in support of our claim.
\begin{exmp}\label{rem}
Let $\mathbb{N}=\displaystyle{\bigcup_{j=1}^{\infty}}D_j$ be a decomposition of $\mathbb{N}$ such that $D_j=\left\{2^{j -1}(2s - 1) : s= 1, 2, \cdots\right\}$. Then each $D_j$ is infinite and $D_i \cap D_j= \phi$ for $i\neq j$. Put $\mathcal{I}$ be the class of all $A\subset \mathbb{N}$ such that $A$ intersects with only a finite numbers of $D_j$'s. Then it is easy to see that $\mathcal{I}$ is an admissible ideal in $\mathbb{N}$. Let $X=\mathbb{R}$, $E=\mathbb{R}^2$ and $P=\left\{(x, y) : x, y\geq 0\right\}\subset \mathbb{R}^2$ be a cone. Define $d : X \times X \rightarrow E$ be such that $d(x, y) = (|x - y|, |x - y|)$. Then $(X, d)$ is a cone metric space. Define a sequence $x=\left\{x_n\right\}$ in $(X,d)$ such that $x_n = \frac{1}{j}$ if $n\in D_j$. Let $r=(r_1, r_2)\in int P$ and $\min(r_1, r_2)=r^*$. Let $(0<<)\varepsilon= (\varepsilon_1, \varepsilon_2)$ be arbitrary and $\min(\varepsilon_1, \varepsilon_2)=\varepsilon^*$. Then by Archimedean property of $\mathbb{R}$, there exists a $l\in\mathbb{N}$ such that $\varepsilon^* >\frac{1}{l}$. Then it is easy to see that $ [-r^*, r^*] \subset I-LIM^r x$, as $\{n\in\mathbb{N} : r + \varepsilon - d(x_n, x^*)\notin int P\}\subset D_1\cup D_2 \cup \cdots \cup D_l$ for any $x^*\in [-r^*, r^*]$. Therefore the sequence defined above is rough $I$-convergent.\\

If possible let this sequence be rough $I^*$-convergent  of roughness degree $r= (r_1, r_2)$ to $x^*=\frac{r^*}{2}$, where $r^*=\min (r_1, r_2)$. Then there exists a set $M=\left\{m_1<m_2<\cdots<m_k<\cdots\right\}\in F(I)$ such that $\left\{x_{m_k}\right\}$ is rough convergent of roughness degree $r$. Now obviously $\mathbb{N}\setminus M = H\in\mathcal{I}$. So there exists a $p\in \mathbb{N}$ such that $H\subset D_1\cup D_2\cup\cdots\cup D_p$ and $D_{p+1}\subset \mathbb{N}\setminus H = M$. Hence we have $x_{m_k}=\frac{1}{p + 1}$ for infinitely many $k$'s. Let us take  $(0<<)r\in E$ in such a way that $r^*=\frac{1}{3(p+1)}$. Let $(0<<)\varepsilon=(\varepsilon_1, \varepsilon_2)\in E$ be chosen such that $(\varepsilon_1, \varepsilon_2)= (\frac{1}{10(p+1)}, \frac{1}{10(p +1)})$. Then $r + \varepsilon - d(x_{m_k}, x^*)\notin int P$ for infinitely many $k$'s. Therefore the sequence is not rough $I^*$-convergent to $x^*=\frac{r^*}{2}$ for this chosen $r$ although the sequence is rough $I$-convergent to $x^*=\frac{r^*}{2}$ for the same $r$. 
\end{exmp}
\begin{rem}
However the sequence in the example \ref{rem} may be rough $I^*$-convergent of different roughness degree to different limit with respect to the same ideal defined in that example. For, if we chose say $r=(5,5)$ and $l=\frac{1}{p + 1}$, then $d(x_{m_k}, l) << r +\varepsilon $ for all $k$'s. Therefore $\{x_n\}_{n\in M}$ is rough convergent and so $\{x_n\}$ is rough $I^*$-convergent.
\end{rem}
Rough $I$-limit and rough $I^*$-limit are same for a sequence $\{x_n\}$ in a cone metric space if the ideal has the property (AP). To prove this we need the following lemma.
\begin{lem}\label{lem}\cite{m}
Let $\left\{A_n\right\}_{n\in\mathbb{N}}$ be a countable family of subsets of $\mathbb{N}$ such that $A_n\in F(I)$ for each $n$, where $F(I)$ is the filter associated with an admissible ideal $I$ with the property (AP). Then there exists a set $B\subset \mathbb{N}$ such that $B\in F(I)$ and the sets $B\setminus A_n$ is finite for all $n$.
\end{lem}
\begin{thm}\label{newone}
If an ideal $I$ has the property (AP) then a sequence $x=\left\{x_n\right\}$ in a cone metric space $(X, d)$ which is rough $I$-convergent of roughness degree $r$ to $x^*\in X$ is also rough $I^*$-convergent of same roughness degree $r$ to $x^*$.
\end{thm}
\begin{proof}
Let $I$ be a ideal in $\mathbb{N}$ which satisfy the property (AP). Let the sequence $x=\left\{x_n\right\}$ be rough $I$-convergent of the roughness degree $r$ to $x^*$. Then for any $(0<<) \varepsilon $ the set $\left\{n\in \mathbb{N} : r + \varepsilon -d(x_n , x^*)\notin int P\right\}\in I$. Therefore the set $\left\{n\in \mathbb{N} : r + \varepsilon - d(x_n, x^*)\in int P\right\}\in F(I)$. Let $(0<<)l\in E$. Now define $A_i=\left\{n\in \mathbb{N} : d(x_n, x^*) < < r + \frac{l}{i}\right\}$, where $i=1, 2,  \cdots$. It is clear that $A_i\in F(I)$ for all $i=1, 2, \cdots$. Since $I$ has the property (AP), therefore there exists a set $B\subset \mathbb{N}$ such that $B\in F(I)$ and $B\setminus A_i$ is finite for $i=1, 2, \cdots$. 
 Now let $(0<<)\varepsilon\in E$, then by lemma \ref{khani} there exists $j\in \mathbb{N}$ such that $\frac{l}{j}<< \varepsilon$. As $B\setminus A_j $ is finite, so there exists a $k=k(j)\in \mathbb{N}$ such that $n\in B\cap A_j$ for all $n\geq k$. Therefore $d(x_n, x^*)<< r +\frac{l}{j} << r + \varepsilon$ for all $n\in B$ and $n\geq k$. Thus the sequence $\left\{x_n\right\}_{n\in B}$ is rough convergent of roughness degree $r$ to $x^*$. Hence the sequence $\left\{x_n\right\}$ is rough $I^*$-convergent of roughness degree $r$ to $x^*$. Hence the theorem. 
\end{proof}
\begin{cor}
Let $\{x_n\}$ be a sequence in a cone metric space $(X,d)$. Then rough $I$-limit set of $\{x_n\}$ equals with rough $I^*$-limit set of $\{x_n\}$ of  roughness degree $r$ if and only if $I$ has the property (AP).
\end{cor}
\begin{proof}
In view of theorem \ref{new} and theorem \ref{newone} the result follows.
\end{proof}
\begin{thm}
If $y=\left\{x_{n_k}\right\}$ be a subsequence of the sequence $x=\left\{x_n\right\}$, then $I-LIM^r x \subset I-LIM^r y$.
\end{thm}
\begin{proof}
If possible let $x^*\in I-LIM^r x$. Then for any $(0<< )\varepsilon \in E$ the set $\{ n\in \mathbb{N} : r +\varepsilon - d(x_n , x^*) \notin int P\}\in I$. Now for the subsequence $y= \{x_{n_k}\}$, as $\{n_k\in \mathbb{N} : r +\varepsilon - d(x_{n_k}, x^*)\notin int P\} \subset \{n\in \mathbb{N} : r +\varepsilon - d(x_n , x^*)\notin int P\}$ and $\{n\in \mathbb{N} : r +\varepsilon - d(x_n , x^*)\notin int P\}\in I$, so $\{n_k\in \mathbb{N} : r +\varepsilon - d(x_{n_k}, x^*)\notin int P\}\in I$. Hence the set $L = \{n_k\in \mathbb{N} : r +\varepsilon - d(x_{n_k}, x^*)\in int P\}\in F(I)$. Let us write $L=\{m_1< m_2<m_3 \cdots \}$. Then $\{x_{m_k}\}_{m_k \in L}$ is a subsequence of $y$. So for the sequence $\{x_{m_k}\}_{{m_k}\in L}$ we have $d(x_{m_k}, x^*) << r +\varepsilon $ and hence $\{x_{m_k}\}_{m_k\in L}$ is rough convergent of roughness degree $r$ to $x^*$. Therefore the sequence $y=\{x_{n_k}\}$ is rough $I^*$-convergence of roughness degree $r$ to $x^*$. So $y=\{x_{n_k}\}$ is also rough $I$-convergent of roughness degree $r$ to $x^*$. Hence $x^*\in I-LIM ^r y$. 
\end{proof}
We now recall following two lemmas from \cite{5}.
\begin{lem}\label{lem1}\cite{5}
Let $(X, d)$ be a cone metric space with normal cone $P$ and normal constant $K$. Then for any $\varepsilon(>0)\in \mathbb{R}$, we can choose $c\in E$ with $c\in int P$ and $K|| c ||< \varepsilon$.
\end{lem}
\begin{lem}\cite{5}\label{last}
Let $(X, d)$ be a cone metric space with normal cone $P$ and normal constant $K$. Then for each $c\in E$ with $0<<c$, there is a $\delta>0$ such that $|| x||< \delta$ implies $c - x\in int P$.
\end{lem}
\begin{thm}
Let $(X, d)$ be a cone metric space with normal cone $P$ and normal constant $K$. Let $I$ be an ideal in $\mathbb{N}$ which has the property (AP). Then a sequence $x=\left\{x_n\right\}$ in $(X, d)$ is rough $I$-convergent of roughness degree $r$ to $x$ if and only if $\left\{d(x_n, x) - r\right\}$ is $I$-convergent to $0$, provided that $\left\{d(x_n, x) - r\right\}$ is a sequence in $P$.
\end{thm}
\begin{proof}
Firstly let us assume that $x=\left\{x_n\right\}$ is rough $I$-convergent of roughness degree $r$ to $x$. Since $I$ has the property (AP) by theorem \ref{newone}, the sequence $x=\left\{x_n\right\}$ is also rough $I^*$-convergent of roughness degree $r$ to $x$. So there exists a set $M=\{m_1<m_2<\cdots<m_k<\cdots\}\in F(I)$ such that the subsequence $\{x_n\}_{n\in M}$ is rough convergent of roughness degree $r$ to $x$. Let $(0<)\varepsilon\in \mathbb{R}$ be given. Then according to lemma \ref{lem1} we have an element $(0<<)c \in E$ with $K\norm{c} < \varepsilon$. Now since the sequence $\{x_n\}_{n\in M}$ is rough convergent of roughness degree $r$ to $X$, so for this $(0<<)c$ we have an element $l\in \mathbb{N}$ such that $d(x_{m_k}, x ) << r + c $ for all $k\geq l$. That is $d(x_{m_k}, x) - r < < c$ for all $k\geq l$. Now as $P$ is normal cone with normal constant $K$, therefore we have $\norm{d(x_{m_k},x) - r}\leq K\norm{c}<\varepsilon$ for all $k\geq l$. Since this is true for any arbitrary $(0<)\varepsilon\in \mathbb{R}$, by lemma \ref{correction} we see that the sequence $\left\{d(x_n, x) - r\right\}_{n\in M}$ converges to $0$. This implies that the sequence $\left\{d(x_n, x) - r\right\}$ is $I^*$-convergent to $0$ and hence it is also $I$-convergent to $0$.\\

Conversely suppose that the sequence $\left\{d(x_n, x) - r\right\}$ is $I$-convergent to $0$. Since the ideal $I$ has the property (AP) and any cone metric space is first countable, therefore the sequence $\left\{d(x_n, x) - r\right\}$ is also $I^*$-convergent to $0$. Thus there exists a set $M=\{m_1< m_2<\cdots<m_k< \cdots\}\in F(I)$ such that $\{d(x_n, x) - r\}_{n\in M}$ is convergent to $0$. Let $c\in E$ with $0<< c$. Then by lemma \ref{last}, there exists a $\delta > 0$, such that $\norm{x}< \delta$ implies $ c- x\in int P \rightarrow (i)$. Now since $\{d(x_n, x) - r\}_{n\in M}$ is convergent to $0$, so for this $\delta$ there exists a $k\in \mathbb{N}$ such that $\norm{d(x_{m_p}, x) - r }< \delta$ for all $p\geq k$. So by $(i)$,\: $ c - (d(x_{m_p}, x) - r) \in int P$ for all $p\geq k$. Thus $d(x_{m_p}, x) << r + c$ for all $p\geq k$. Therefore $\left\{n\in \mathbb{N} : r + c - d(x_n , x)\notin int P\right\}\subset \mathbb{N} \setminus M \cup \left\{m_1<m_2<\cdots< m_{k -1}\right\}$ and hence $\left\{n \in \mathbb{N} : r + c - d(x_n, x) \notin int P\right\}\in I$. Therefore $\left\{x_n\right\}$ is rough $I$-convergent of roughness degree $r$ to $x$.  
\end{proof}
\begin{thm}
Let $(X, d)$ be a cone metric space with normal cone $P$ and normal constant $K$. Also let $\left\{x_n\right\}$ and $\left\{y_n\right\}$ be two sequence in $(X, d)$ rough $I$-convergent of roughness degree $\frac{1}{4K + 2} r$ to $x$ and $y$ respectively. Then the sequence $\left\{z_n\right\}$ in $E$ is rough $I$-convergent to $d(x, y)$ of roughness degree $||r||$ where $z_n=d(x_n, y_n)$ for all $n\in \mathbb{N}$.
\end{thm}
\begin{proof}
Let $\varepsilon >0$ be given and $x\in int P$. Then $c=\frac{ \varepsilon x}{2\norm{x}(4K + 2)}\in int P$. Obviously $\norm{c}<\frac{\varepsilon}{2(4K +2)}$. Now as $\left\{x_n\right\}$ and $\left\{y_n\right\}$ both are rough $I$-convergent of same roughness degree $\frac{1}{4K + 2} r$ to $x$ and $y$ respectively, therefore for $(0<<)c\in E$ the sets $A_1=\left\{n\in \mathbb{N} : c + \frac{1}{4K + 2} r - d(x_n, x)\notin int P\right\}$ and $A_2=\left\{n\in \mathbb{N} : c + \frac{1}{4K + 2} r - d(y_n, y)\notin int P\right\}$ both belongs to $I$. Therefore the set $A_1^\complement=\mathbb{N}\setminus A_1$ and $A_2^\complement=\mathbb{N}\setminus A_2$ both belongs to $F(I)$. Therefore $M=A_1^\complement \cap A_2^\complement\in F(I)$. Let us choose an element $m\in A_1^\complement \cap A_2^\complement$. So $(c + \frac{1}{4K + 2} r) - d(x_m, x)\in int P\rightarrow (i)$ and $(c + \frac{1}{4K + 2} r) - d(y_m, y) \in int P \rightarrow (ii)$. From $(i)$ and $(ii)$ we get $(c + \frac{1}{4K + 2} r) - d(x_m, x) + (c + \frac{1}{4K + 2} r) - d(y_m, y) =2 (c + \frac{1}{4K + 2} r)  - (d(x_m, x) + d(y_m, y))\in int P\rightarrow(iii)$.\\
Again $d(x, y)\leq d(x_m, x) + d(x_m, y)$ i.e., $d(x_m, x) + d(x_m, y) - d(x, y)\in P\rightarrow(iv)$. Also as $d(x_m, y) \leq d(x_m, y_m) + d(y_m, y)$, so $d(x_m , y_m) + d(y_m, y) - d(x_m, y)\in P\rightarrow(v)$. Thus from $(iv)$ and $(v)$ we get $d(x_m,x) + d(y_m, y) + d(x_m, y_m) - d(x, y)\in P\rightarrow (vi)$. Also as $d(x, y_m)\leq d(x, y) + d(y, y_m)$ i.e., $d(x, y) + d(y, y_m) - d(x, y_m)\in P$ and as $d(x_m, y_m)\leq d(x_m ,x) + d(x, y_m)$ i.e., $ d(x_m ,x) + d(x, y_m)- d(x_m, y_m)\in P$ so their sum also belongs to $P$. That is $d(x, y) + d(y, y_m) + d(x_m,x) - d(x_m, y_m) \in P\rightarrow(vii)$. Now from $(iii)$ and $(vii)$ we get $2(c + \frac{1}{4k + 2} r) + d(x, y) - d(x_m, y_m) \in int P\rightarrow(viii)$. Again from $(iii)$ and $(vi)$ we have $2(c + \frac{1}{4K + 2} r) + d(x_m, y_m) - d(x, y)\in int P$, i.e., $4(c + \frac{1}{4K + 2} r) - ( 2(c + \frac{1}{4K + 2} r) +d(x, y) - d(x_m, y_m)) \in int P$. This implies that $ 2(c + \frac{1}{4K + 2} r) +d(x, y) - d(x_m, y_m)<< 4(c + \frac{1}{4K + 2} r)$. Also from $(viii)$ we have $2(c + \frac{1}{4K + 2} r) + d(x, y)- d(x_m, y_m)\in int P$. Again we have $0<<  4(c + \frac{1}{4K + 2} r)$. Now as $P$ is normal therefore $|| 2(c + \frac{1}{4K + 2} r) +d(x, y) - d(x_m, y_m)|| \leq K ||4(c + \frac{1}{4K + 2} r)||\rightarrow(ix)$. Also $|| d(x, y) - d(x_m, y_m)|| = || d(x, y) - d(x_m, y_m) + 2(c + \frac{1}{4K + 2} r) - 2(c + \frac{1}{4K + 2} r)||\leq ||d(x, y) - d(x_m, y_m) + 2(c + \frac{1}{4K + 2} r)|| + 2 |
|(c + \frac{1}{4K + 2} r)||\rightarrow (x)$. Thus from $(x)$ using $(ix)$ we get $|| d(x, y) - d(x_m, y_m)|| \leq 4K ||(c + \frac{1}{4K + 2} r)|| + 2 ||(c + \frac{1}{4K + 2} r)|| = (4K + 2)||(c + \frac{1}{4K + 2} r)||\leq (4K + 2)||c|| + (4K + 2)\frac{1}{4K + 2} || r|| = \frac{\varepsilon}{2} + || r||< \varepsilon + || r|| \rightarrow(xi)$. Since the inequality in $(xi)$ holds for any arbitrary $m\in M$, therefore $\left\{n\in \mathbb{N} : || d(x_n, y_n ) - d(x, y) || \geq \varepsilon + || r||\right\}\subset \mathbb{N}\setminus M$. Hence the set  $\left\{n\in \mathbb{N} : || d(x_n, y_n ) - d(x, y) || \geq \varepsilon + || r||\right\}\in I$. Hence the sequence $\left\{z_n\right\}$ in $E$ is rough $I$-convergent to $d(x, y)\in E$ of roughness degree $|| r||$.
\end{proof}
 
\end{document}